\documentclass[12pt]{amsart}
\usepackage{amsmath}
\usepackage{mathtools}
\usepackage{amssymb}
\usepackage{amsthm}
\usepackage{graphicx}
\usepackage{enumerate}
\usepackage[mathscr]{eucal}
\usepackage[pagewise]{lineno}
\usepackage{tikz}
\usetikzlibrary{decorations.text,calc,arrows.meta}
\theoremstyle{plain}

\usepackage[english]{babel}
\DeclareMathOperator{\conv}{conv}
\DeclareMathOperator{\sgn}{sgn}

\newtheorem{theorem}{Theorem}
\newtheorem{cor}[theorem]{Corollary}

\newtheorem{prop}[theorem]{Proposition}
\newtheorem{remark}[theorem]{Remark}
\theoremstyle{definition}
\newtheorem{definition}[theorem]{Definition}

\usepackage{hyperref}
\hypersetup{
    colorlinks=true,
    linkcolor=blue,
    filecolor=magenta,      
    urlcolor=cyan,
}
\usepackage[pagewise]{lineno}
\usepackage{hyperref}
\begin{document}

\title[Symmetry of BJ Orthogonality and Geometry of $\mathbb{B}(\ell_\infty^n,\ell_1^m)$]{Point-wise Symmetry of Birkhoff-James Orthogonality and Geometry of $\mathbb{B}(\ell_\infty^n,\ell_1^m)$}

\author[Bose]{Babhrubahan Bose}
                \newcommand{\acr}{\newline\indent}
                

\address[Bose]{Department of Mathematics\\ Indian Institute of Science\\ Bengaluru 560012\\ Karnataka \\INDIA\\ }
\email{babhrubahanb@iisc.ac.in}

 \thanks{The research of Babhrubahan Bose is funded by PMRF research fellowship under the supervision of Professor Apoorva Khare and Professor Gadadhar Misra.}

\begin{abstract}
We study the relationship between the point-wise symmetry of Birkhoff-James orthogonality and the geometry of the space of operators $\mathbb{B}(\ell_\infty^n,\ell_1^m)$. We show that any non-zero left-symmetric point in this space is a smooth point. We also show that for $n\geq4$, any unit norm right-symmetric point of this space is an extreme point of the closed unit ball. This marks the first step towards characterizing the extreme points of these unit balls and finding the Grothendieck constants $G(m,n)$ using Birkhoff-James orthogonality techniques.
\end{abstract}

\maketitle   
   
\section*{Introduction}
In recent times, Birkhoff-James orthogonality and its point-wise symmetry has been used to understand the geometry of a normed linear space. Characterization of Birkhoff-James orthogonality and its local symmetry has been done for finite-dimensional $\ell_p$ spaces in \cite{CSS}, while that for the sequence spaces $\ell_p$ and the function spaces $L_p$ have been done in \cite{usseq}, \cite{me}. In these articles, these characterizations have been used to understand the geometry of the underlying spaces by describing the smooth points and the onto isometries of the spaces. In this article, we use this idea to establish a relationship between the point-wise symmetry of Birkhoff-James orthogonality and the geometry of the space of operators from $\ell_\infty^n$ into $\ell_1^m$ over $\mathbb{R}$ or $\mathbb{C}$, denoted by $\mathbb{B}(\ell_\infty^n,\ell_1^m)$. We show that any non-zero left-symmetric point of this space is a smooth point and any unit norm right-symmetric point of this space is an extreme point of the closed unit ball.\par
Recall that these extreme points play a crucial role in understanding the \textit{Grothendieck constant} for a given pair of natural numbers $(m,n)$, given by
\begin{align}\label{grothendieck}
    G^\mathbb{K}(m,n):=\sup_{\substack{\|T\|_{op}=1,\\
    \|x_i\|_2=\|y_j\|_2=1}}
    ~\left| \sum\limits_{i=1}^m \sum\limits_{j=1}^n a_{ij}\langle x_i,y_j\rangle\right|,
\end{align}
where $T=[a_{ij}]_{m\times n}\in\mathbb{B}(\ell_\infty^n,\ell_1^m)$ over $\mathbb{K}=\mathbb{R}$ or $\mathbb{C}$. The Grothendieck constant $G^\mathbb{K}$ for the field $\mathbb{K}$ can then be obtained as the supremum of all $G^\mathbb{K}(m,n)$, where $m$ and $n$ vary over the natural numbers.\par

The Grothendieck constant emerged from a celebrated result by Grothendieck \cite{groth}, which in effect says that $G^\mathbb{K}(m,n)$ (defined in \eqref{grothendieck}) is uniformly bounded over both $\mathbb{R}$ and $\mathbb{C}$. This theorem/inequality of Grothendieck, including determining the exact constant (or sharper bounds for it \cite{Krivine}), has been the focus of tremendous research in the past few decades, including in Banach space theory, C* algebra theory, operator theory, physics, and computer science (see e.g. the survey \cite{Khot}). We refer the reader to the comprehensive and authoritative survey by Pisier \cite{Pisier} and the memoir by Blei \cite{blei} for more on the Grothendieck inequality.

The goal of this short note is to approach the Grothendieck inequality via the framework of Birkhoff-James orthogonality. Since the right hand side of \eqref{grothendieck} is the supremum of a convex function taken over a convex set, the supremum is attained at the extreme points of the convex set, viz, the closed unit ball of $\mathbb{B}(\ell_\infty^n,\ell_1^m)$. Hence characterization of these extreme points, which we call the \textit{extreme contractions} (following \cite{iwa}), would allow us to find the Grothendieck constant for the given pair $(m,n)$.\par
In this spirit, the first step towards better bounding the Grothendieck constant would be to characterize the extreme contractions. This has been done for $m=1$, $n=1$ and $(m,n)=(2,2),(3,3)$ in \cite{gm} and \cite{lima}. In this note we focus on the case $n \geq 4$, with the ground field either $\mathbb{R}$ or $\mathbb{C}$, where our main result provides a class of extreme contractions:
\begin{theorem}\label{main}
Let $n\geq4$. If $T\in\mathbb{B}(\ell_\infty^n,\ell_1^m)$ over $\mathbb{R}$ or $\mathbb{C}$ is a right-symmetric point and has norm 1, then $T$ is an extreme contraction.
\end{theorem}
\par
For proving this result, we first study Birkhoff-James orthogonality in the Banach space $\ell_\infty^n\otimes\mathbb{X}$ for any Banach space $\mathbb{X}$. Using the results obtained therefrom, we prove the relationship between the point-wise symmetry of Birkhoff-James orthogonality and the geometry of the space of operators mentioned before. In the sections \ref{2} and \ref{3} we do these analyses for the real case, while the fourth and final section is dedicated towards generalizing the results obtained so far in case of Banach spaces over the complex field.

\section{Notations and terminologies}\label{1}
Let us establish the relevant notations and terminologies to be used throughout the article. For a Banach space $\mathbb{X}$, let $\mathbb{X}^*$ denote the continuous dual of it and define the support functional of a non-zero element $x\in\mathbb{X}$ to be any $f\in\mathbb{X}^*$ such that
\begin{align*}
    \|f\|=1,~ f(x)=\|x\|.
\end{align*}
A non-zero element $x\in\mathbb{X}$ is said to be smooth if it has a unique support functional. \par
Given two elements $x,y\in\mathbb{X}$, $x$ is defined to be \textit{Birkhoff-James orthogonal} to $y$ \cite{B}, denoted by $x\perp_By$ if 
\[\|x+\lambda y\|\geq\|x\|,~\textit{for every scalar}~\lambda.\]
James proved in \cite{james} that $x\perp_By$ if and only if $f(y)=0$ for some support functional $f$ of $x$. In the same article, he proved that a non-zero point $x\in\mathbb{X}$ is smooth if and only if Birkhoff-James orthogonality is right additive at $x$, i.e., for any $y,z\in\mathbb{X}$,
\begin{align*}
    x\perp_By,~x\perp_Bz~\Rightarrow~x\perp_B(y+z).
\end{align*}
James proved in \cite{james2} that in a normed linear space of dimension 3 or more, Birkhoff-James orthogonality is symmetric if and only if the space is an inner product space. However, the importance of studying the point-wise symmetry of Birkhoff-James orthogonality in describing the geometry of normed linear spaces has been illustrated in \cite[Theorem 2.11]{CSS}, \cite[Corollary 2.3.4]{Sain}. Let us recall the following definition in this context from \cite{Sain2}, which will play an important part in our present study.
\begin{definition}
An element $x$ of a normed linear space $\mathbb{X}$ is said to be \textit{left-symmetric} (\textit{resp. right-symmetric}) if 
\begin{align*}
    x\perp_By\;\Rightarrow\; y\perp_Bx~~(\textit{resp.~}y\perp_Bx\;\Rightarrow\;x\perp_By),
\end{align*}
for every $y\in \mathbb{X}$.
\end{definition}
Note that by the term \textit{point-wise symmetry of Birkhoff-James orthogonality}, we refer to the left-symmetric and the right-symmetric points of a given normed linear space.\par
A \textit{semi-inner product} on a real vector space $\mathbb{V}$ is defined to be a map $[\cdot,\cdot]:\mathbb{V}\times\mathbb{V}\to\mathbb{K}$ such that for $x,y,z\in\mathbb{V}$ and $\lambda\in\mathbb{K}$,
\begin{enumerate}
    \item $[x,x]\geq0$ with equality if and only if $x=0$.
    \item $[y,x]+\lambda[z,x]=[y+\lambda z,x]$.
    \item $[x,\lambda y]=\overline{\lambda} [x,y]$
    \item $|[x,x]|^2\leq[x,x][y,y]$.
\end{enumerate}
A \textit{semi-inner product} on a Banach space $\mathbb{X}$ is a map $[\cdot,\cdot]:\mathbb{X}\times\mathbb{X}\to\mathbb{K}$ satisfying the above four properties along with $[x,x]=\|x\|^2$ for every $x\in\mathbb{X}$. Construction of a semi-inner product on $\mathbb{X}$ requires a map $\Psi:\mathbb{K}\mathbb{P}\mathbb{X}\to S_{\mathbb{X}^*}$ such that $\Psi([x])$ is the support functional of some $x_0\in [x]\cap S_\mathbb{X}$, where $\mathbb{K}\mathbb{P}\mathbb{X}$ denotes the $\mathbb{K}$-projective space of $\mathbb{X}$ \footnote{Recall, this is the set of equivalence classes in $\mathbb{X}$ under the relation of multiplication by nonzero scalars in $\mathbb{K}$ -- i.e., the set of lines in $\mathbb{X}$.} and $[x]$ denotes the equivalence class of $x\in\mathbb{X}\setminus \{0\}$ in $\mathbb{K}\mathbb{P}\mathbb{X}$. Note that the element $x_0$ is uniquely determined by the map $\Psi$ and the element $x\in\mathbb{X}\setminus \{0\}$. The semi-inner product can then be constructed as
\begin{align*}
    [y,x]:=\overline{\lambda}\left(\Psi([x])\right)(y);~x=\lambda x_0,~(\Psi([x]))(x_0)=1,~x,y\in\mathbb{X}.
\end{align*}
Note that a non-zero point $x\in\mathbb{X}$ is smooth if and only if 
\begin{align*}
    [y,x]_1=[y,x]_2,~\textit{for every}~y\in\mathbb{X},
\end{align*}
for every pair $[\cdot,\cdot]_1$, $[\cdot,\cdot]_2$ of semi-inner products on $\mathbb{X}$. Also, for $x,y\in\mathbb{X}$, $x\perp_By$ if and only if $[y,x]=0$ for some semi-inner product $[\cdot,\cdot]$ on $\mathbb{X}$.\par
Let us denote the space $\ell_\infty^n\otimes\mathbb{X}$ by $\mathbb{X}_\infty^n$. Then $\mathbb{X}_\infty^n$ is the vector space of all $\mathbb{X}$-valued sequences of length $n$ with the norm defined as
\begin{align*}
    \|(x_1,x_2,\dots,x_n)\|:=\max\limits_{1\leq i\leq n}\|x_i\|,~x_i\in\mathbb{X}.
\end{align*}
Also, denote $\ell_1^n\otimes\mathbb{X}$ by $\mathbb{X}_1^n$, i.e., 
\begin{align*}
    \mathbb{X}_1^n:=\{(x_1,x_2,\dots,x_n):x_i\in\mathbb{X}\},~\|(x_1,x_2,\dots,x_n)\|:=\sum\limits_{i=1}^n\|x_i\|,~x_i\in\mathbb{X}.
\end{align*}
The Banach space of all bounded linear operators between two Banach spaces $\mathbb{X}$ and $\mathbb{Y}$ equipped with the operator norm is denoted by $\mathbb{B}(\mathbb{X},\mathbb{Y})$. We call an extreme point of the closed unit ball of this space, \textit{extreme contraction}. Also, the extreme points of the closed unit ball of any Banach space $\mathbb{X}$ are denoted simply by \textit{extreme points of $\mathbb{X}$}. An operator $T\in\mathbb{B}(\mathbb{X},\mathbb{Y})$ is said to attain norm at $x\in\mathbb{X}$ if $\|x\|=1$ and $\|Tx\|=\|T\|$. The set of all points where an operator $T$ attains norm is denoted by $M_T$, i.e.,
\[M_T:=\{x\in\mathbb{X}:\|x\|=1,~\|Tx\|=\|T\|\}.\]

\section{Birkhoff-James orthogonality in $\mathbb{X}_\infty^n$}\label{2}
Throughout this section and the next one, all the spaces are to be considered over $\mathbb{R}$. Begin by observing that the dual of $\mathbb{X}_\infty^n$ is ${\mathbb{X}^*}_1^n$, with the functional $\Psi_{(f_1,f_2,\dots f_n)}$ corresponding to $(f_1,f_2,\dots,f_n)\in{\mathbb{X}^*}_1^n$ given by
\begin{align*}
    \Psi_{(f_1,f_2,\dots f_n)}(x_1,x_2,\dots,x_n):=\sum\limits_{i=1}^nf_i(x_i),~(x_1,x_2,\dots,x_n)\in\mathbb{X}_\infty^n.
\end{align*}
We now characterize the support functional of any non-zero element of $\mathbb{X}_\infty^n$. 
\begin{prop}\label{support}
Given $x=(x_1,x_2,\dots,x_n)\in\mathbb{X}_\infty^n$ non-zero, $(f_1,f_2,\dots,f_n)\in{\mathbb{X}^*}_1^n$ is a support functional of $x$ if and only if 
\begin{align*}
    f_i=
    \begin{cases}
    \lambda_ig_i,~\|x_i\|=\|x\|,\\
    0,~\|x_i\|<\|x\|,
    \end{cases}
\end{align*}
where $g_i$ is a support functional of $x_i$ for every $1\leq i\leq n$ and $\sum\limits_i\lambda_i=1$.
\end{prop}
\begin{proof}
To prove the sufficiency, note that if $(f_1,f_2,\dots,f_n)$ satisfies the given condition, then
\begin{align*}
    \|(f_1,f_2,\dots,f_n)\|=\sum\limits_{i=1}^n\|f_i\|=\sum\limits_{i=1}^n\lambda_i\|g_i\|=\sum\limits_{i=1}^n\lambda_i=1,
\end{align*}
and
\begin{align*}
    \Psi_{(f_1,f_2,\dots,f_n)}(x_1,x_2,\dots,x_n)=\sum\limits_{i=1}^n\lambda_ig_i(x_i)=\sum\limits_{\|x_i\|=\|x\|}\lambda_ig_i(x_i)=\sum\limits_{\|x_i\|=\|x\|}\lambda_i\|x_i\|=\|x\|,
\end{align*}
since $\lambda_i=0$ if $\|x_i\|\neq\|x\|$.\par
For the necessity, note that if $\Psi_{(f_1,f_2,\dots,f_n)}$ is a support functional of $(x_1,x_2,\dots,x_n)$, then 
\begin{align*}
    \|x\|=\sum\limits_{i=1}^nf_i(x_i)\leq\sum\limits_{i=1}^n\|f_i\|\|x_i\|\leq \sum\limits_{i=1}^n\|f_i\|\|x\|=\|x\|\|(f_1,f_2,\dots,f_n)\|.
\end{align*}
Since $\|(f_1,f_2,\dots,f_n)\|=1$, equality must hold in all the inequalities giving $f_i=0$ if $\|x_i\|<\|x\|$ and $f_i(x_i)=\|f_i\|\|x_i\|$ if $\|x_i\|=\|x\|$. Hence, $f_i=\lambda_ig_i$ for some support functional $g_i$ of $x_i$ and $\lambda_i\geq0$ if $\|x_i\|=\|x\|$. Finally,
\begin{align*}
    \sum\limits_{i=1}^n\lambda_i=\sum\limits_{\|x_i\|=\|x\|}\lambda_i=\sum\limits_{\|x_i\|=\|x\|}\lambda_i\|g_i\|=\sum\limits_{\|x_i\|=\|x\|}\|f_i\|=\sum\limits_{i=1}^n\|f_i\|=1.
\end{align*}
\end{proof}
\begin{remark}
If $\mathbb{X}=\mathbb{R}$, then $\mathbb{X}_\infty^n$ is in fact $\ell_\infty^n$, where the support functional of a non-zero $(x_1,x_2,\dots,x_n)\in\ell_\infty^n$ is characterized by any convex combination of $\{\overline{\sgn(x_i)}e_i:|x_i|=\|x\|\}$ as can be verified both from Proposition \ref{support} and direct computations.
\end{remark}
We are now ready to characterize Birkhoff-James orthogonality, smoothness and right-symmetry in $\mathbb{X}_\infty^n$. It should be mentioned that for $\mathbb{X}=\mathbb{R}$ case, these results give the corresponding characterizations for $\ell_\infty^n$ and elementary computations would reveal that our results indeed conform with the characterizations in \cite{CSS} by Chattopadhyay, Sain and Senapati. \par
We begin with characterizing Birkhoff-James orthogonaliy in $\mathbb{X}_\infty^n$.
\begin{theorem}\label{ortho}
Given two elements $x=(x_1,x_2,\dots,x_n)$ and $y=(y_1,y_2,\dots,y_n)$ of $\mathbb{X}_\infty^n$, $x\perp_By$ if and only if either of the following two conditions holds
\begin{enumerate}
    \item $x_i\perp y_i$ for some $1\leq i\leq n$ such that $\|x_i\|=\|x\|$.
    \item $[y_i,x_i]_1[y_j,x_j]_2<0$ for some $1\leq i<j\leq n$ and $[\cdot,\cdot]_1,[\cdot,\cdot]_2$ two semi-inner products on $\mathbb{X}$, where $\|x_i\|=\|x_j\|=\|x\|$.
\end{enumerate}
\end{theorem}
\begin{proof}
We begin with the sufficiency. If condition 1 holds then by James' characterization of Birkhoff-James orthogonality, there exists a support functional $f$ of $x_i$ such that $f(y_i)=0$. Define $\Gamma:\mathbb{X}_\infty^n\to\mathbb{R}$ given by
\begin{align*}
    \Gamma((z_1,z_2,\dots,z_n)):=f(z_i),~(z_1,z_2,\dots,z_n)\in\mathbb{X}_\infty^n.
\end{align*}
Then by Proposition \ref{support}, $\Gamma$ is a support functional of $x$ and $\Gamma(y)=0$ giving $x\perp_By$. If condition 2 holds, then define $f_i,f_j:\mathbb{X}\to\mathbb{R}$ given by
\begin{align*}
    f_i(z):=[z,x_i]_1,~f_j(z):=[z,x_j]_2,~z\in\mathbb{X}.
\end{align*}
Then $f_i$ and $f_j$ are support functionals of $x_i$ and $x_j$ respectively. Further, since $f_i(y_i)$ and $f_j(y_j)$ are of opposite signs, there exists $\lambda\in(0,1)$ such that $\lambda f_i(y_i)+(1-\lambda)f_j(y_j)=0$. Define $\Gamma':\mathbb{X}_\infty^n\to\mathbb{R}$ as
\begin{align*}
    \Gamma'((z_1,z_2,\dots,z_n)):=\lambda f_i(z_i)+(1-\lambda)f_j(z_j).
\end{align*}
Then again by Proposition \ref{support}, $\Gamma'$ is a support functional of $x$ and $\Gamma'(y)=0$ giving $x\perp_By$.
\par
For proving the necessity, consider $(f_1,f_2,\dots,f_n)\in{\mathbb{X}^*}_1^n$ such that $\Psi_{(f_1,f_2,\dots,f_n)}$ is a support functional of $x$ and $\Psi_{(f_1,f_2,\dots,f_n)}(y)=0$. By Proposition \ref{support}, there exists $\lambda_i\in[0,1]$ and $g_i$ a support functional of $x_i$ such that
\begin{align*}
    \sum\limits_{\|x_i\|=\|x\|}\lambda_ig_i(y_i)=0.
\end{align*}
Therefore, $g_i(y_i)=0$ for some $1\leq i\leq n$, i.e., $x_i\perp_By_i$ or $g_i(y_i)$ and $g_j(y_j)$ are of opposite signs for some $1\leq i<j\leq n$ giving $g_i(y_i)g_j(y_j)<0$. Finding semi-inner products $[\cdot,\cdot]_1$ and $[\cdot,\cdot]_2$ of $\mathbb{X}$ such that 
\begin{align*}
    [z,x_i]=g_i(z)~\textit{and}~[z,x_j]=g_j(z),~\textit{for every}~z\in\mathbb{X},
\end{align*}
we get condition 2.
\end{proof}
This result allows us to characterize smoothness in this space. 
\begin{cor}\label{smoothinf}
A point $x=(x_1,x_2,\dots,x_n)\in\mathbb{X}_\infty^n$ is smooth if and only if there exists a unique $1\leq i\leq n $ such that $\|x_i\|=\|x\|$ and if $x_i$ is smooth.
\end{cor}
\begin{proof}
For the necessity, note that if $\|x_i\|=\|x_j\|=\|x\|$ for some $1\leq i<j\leq n$, then by Proposition \ref{support}, we can find more than one support functional of $x$. Also if $\|x_i\|=\|x\|$ and $x_i$ is not smooth, then again by Proposition \ref{support}, we can find more than one support functional of $x$.\par
For the sufficiency, note that if the aforesaid condition holds, then by Theorem \ref{ortho}, $x\perp_By$ for some $y=(y_1,y_2,\dots,y_n)$ if and only if $x_i\perp_By_i$, where $\|x_i\|=\|x\|$. Hence clearly, for $y=(y_1,y_2,\dots,y_n),~z=(z_1,z_2,\dots,z_n)$, 
\begin{align*}
    x\perp_By,~x\perp_Bz\Leftrightarrow x_i\perp_By_i,~x_i\perp_Bz_i\Rightarrow x_i\perp_B(y_i+z_i)\Leftrightarrow x\perp_B(y+z).
\end{align*}
\end{proof}
We now prove a necessary condition for an element of $\mathbb{X}_\infty^n$ to be right-symmetric.
\begin{theorem}\label{rightsym}
If $x=(x_1,x_2,\dots,x_n)\in\mathbb{X}_\infty^n$ is a right-symmetric point of $\mathbb{X}_\infty^n$, then $\|x_i\|=\|x\|$ for every $1\leq i\leq n$.
\end{theorem}
\begin{proof}
For the sake of contradiction, let $x$ be a right-symmetric point of $\mathbb{X}_\infty^n$ and without loss of generality, assume that $\|x_n\|<\|x\|$. Consider $y=(y_1,y_2,\dots,y_n)\in\mathbb{X}_\infty^n$ such that $\|y_i\|<\|y\|$ for every $i<n$ and $y_n\perp_Bx_n$. Now, observe that for any $z\in\mathbb{X}$ non-zero,
\begin{align*}
    J(z):=\{f\in\mathbb{X}^*:\|f\|=1,~f(z)=\|z\|\},
\end{align*}
is a convex set and therefore is connected. Further, $J(z)$ is a closed subset of the closed unit ball of $\mathbb{X}^*$ under the weak* topology and hence is weak* compact by the Banach-Alaoglu theorem \cite{BAT}. Therefore,
\begin{align*}
    I(w):=\{f(w):f\in J(z)\},~w\in\mathbb{X},
\end{align*}
is a compact connected subset of $\mathbb{R}$ and therefore is a finite closed interval. Hence if $z\not\perp_Bw$, then by James' characterization, $I(w)$ is contained either in the positive real line or in the negative real line. Thus, if $z\not\perp_Bw$, then $[z,w]$ does not change sign as $[\cdot,\cdot]$ varies over all the semi-inner products on $\mathbb{X}$.\par
Therefore, we can choose $y_i\in\mathbb{X}$ such that $[y_i,x_i]>0$ for every $i<n$ and every semi-inner product $[\cdot,\cdot]$ on $\mathbb{X}$. For this $y$, we get by Theorem \ref{ortho}, $y\perp_Bx$ and $x\not\perp_By$ contradicting the right-symmetry of $x$.
\end{proof}
The following remark will be required later.
\begin{remark}\label{choice}
If $x=(x_1,x_2,\dots,x_n)\in\mathbb{X}_\infty^n$ is such that $\|x_i\|<\|x\|$, then given any $\epsilon>0$, we can find $y=(y_1,y_2,\dots,y_n)\in\mathbb{X}_\infty^n$ satisfying
\begin{align*}
    y\perp_Bx~\textit{and}~x\not\perp_By,
\end{align*}
such that $\|y_i\|=\|y\|$ and $\|y_j\|<\epsilon$ for every $j\neq i$.
\end{remark}

\section{Point-wise symmetry of Birkhoff-James orthogonality and the geometry of $\mathbb{B}(\ell_\infty^n,\ell_1^m)$}\label{3}
In this section, we show that any non-zero left-symmetric point of $\mathbb{B}(\ell_\infty^n,\ell_1^m)$ over $\mathbb{R}$ is smooth and any unit norm right-symmetric point of the space is an extreme point of the closed unit ball. We begin by characterizing the smooth points of the space of operators between two finite-dimensional Banach spaces.
\begin{theorem}\label{smooth}
Given finite-dimensional real Banach spaces $\mathbb{X}$ and $\mathbb{Y}$, an operator $T\in\mathbb{B}(\mathbb{X},\mathbb{Y})$ is smooth if and only if $M_T=\{x_0,-x_0\}$ for some $x_0\in\mathbb{X}$ and $Tx_0$ is a smooth point of $\mathbb{Y}$.
\end{theorem}
\begin{proof}
For proving the sufficiency, note that if $M_T=\{x_0,-x_0\}$, then by \cite[Corollary 2.2.1]{Sain2}, $T\perp_BT'$ for some $T'\in\mathbb{B}(\mathbb{X},\mathbb{Y})$ if and only if $Tx_0\perp_BT'x_0$. Hence by smoothness of $Tx_0$, we get, for any $T',T''\in\mathbb{B}(\mathbb{X},\mathbb{Y})$, 
\begin{align*}
    T\perp_BT',~T\perp_BT''&\Leftrightarrow Tx_0\perp_BT'x_0,~Tx_0\perp_BT''x_0 \\
    &\Rightarrow Tx_0\perp_B(T'x_0+T''x_0)\Leftrightarrow T\perp_B(T'+T'').
\end{align*}
For proving the necessity, note that if $x\in M_T$ and $Tx\in\mathbb{Y}$ is not smooth, then there exist $f,g$ distinct support functionals of $Tx$ giving two distinct support functionals $\Psi$ and $\Phi$ of $T$ via
\begin{align*}
    \Psi(T'):=f\left(T'x\right),~\Phi(T'):=g\left(T'x\right),~T'\in\mathbb{B}(\mathbb{X},\mathbb{Y}).
\end{align*}
Also if $M_T$ contains two linearly independent points $x$ and $y$, then consider support functionals $f$ and $g$ (not necessarily distinct) of $Tx$ and $Ty$ respectively to get two distinct support functionals $\Psi$ and $\Phi$ of $T$ via
\begin{align*}
    \Psi(T'):=f\left(T'x\right),~\Phi(T'):=g\left(T'y\right),~T'\in\mathbb{B}(\mathbb{X},\mathbb{Y}).
\end{align*}
\end{proof}
We now prove the relationship between the left-symmetric points and smoothness in $\mathbb{B}(\ell_\infty^n,\ell_1^m)$.
\begin{theorem}\label{leftsym}
A non-zero element $T\in\mathbb{B}(\ell_\infty^n,\ell_1^m)$ over $\mathbb{R}$ is smooth if it is left-symmetric.
\end{theorem}
\begin{proof}
Let $x\in M_T$ and assume that $Tx$ is not a left-symmetric point of $\ell_1^m$. Then find $y\in\mathbb{Y}$ such that $Tx\perp_By$ and $y\not\perp_BTx$. Construct $T'\in\mathbb{B}(\ell_\infty^n,\ell_1^m)$ such that $M_{T'}=\{x,-x\}$ and $T'x=y$. Hence by \cite[Corollary 2.2.1]{Sain2}, we get $T\perp_BT'$ and $T'\not\perp_BT$. Thus $Tx$ must be a smooth point of $\mathbb{Y}$, whenever $x\in M_T$.\par
Now, if $M_T$ contains more than two points, then $M_T$ must contain at least two extreme points. Since $x\perp_By$ for any two linearly independent extreme points $x,y\in M_T$, we can find $T'\in\mathbb{B}(\mathbb{X},\mathbb{Y})$ such that $M_{T'}=\{x,-x\}$, $T'x=Tx$ and $T'y=0$ for every extreme point $y\in M_T$ linearly independent to $x$. Hence again by \cite[Corollary 2.2.1]{Sain2}, $T\perp_BT'$ and $T'\not\perp_BT$. \par
Thus $M_T=\{x,-x\}$ and $Tx$ is a left-symmetric point of $\ell_1^m$. However, from \cite{CSS}, we get that every left-symmetric point of $\ell_1^m$ is smooth and hence $T$ must be smooth by Theorem \ref{smooth}.
\end{proof}
We now come to our main result, where we relate the notion of right-symmetry of Birkhoff-James orthogonality with extreme contractions. Recall from \cite{iwa} that these are simply defined to be the extreme points of the closed unit ball in $\mathbb{B}(\ell_\infty^n,\ell_1^m)$.
\begin{proof}[Proof of Theorem \ref{main} for the real case]\hfill
\\
Let us denote $\ell_\infty^n$ and $\ell_1^m$ by $\mathbb{X}$ and $\mathbb{Y}$ respectively. Begin by observing that if $M_T$ contains $n$ linearly independent extreme points of $\ell_\infty^n$, then $T$ is an extreme contraction since if $T=\frac{1}{2}(T_1+T_2)$ for some $T_1,T_2\in\mathbb{B}(\mathbb{X},\mathbb{Y})$ having norm 1, then $Tx=\frac{1}{2}(T_1x+T_2x)$ for every extreme point $x$ of $X$ contained in $M_T$, giving $T_1x=T_2x=Tx$. Since there are $n$-many linearly independent $x$ in $M_T$, clearly $T$, $T_1$ and $T_2$ agree on $\mathbb{X}$ proving $T$ to be an extreme contraction.\par
Now, assume the contrary. As $T$ is not an extreme contraction, $T$ cannot attain norm at more than $n-1$ linearly independent extreme points. Let $T$ attain norm at $x_1,x_2,\dots,x_k$, which are linearly independent extreme points of $\mathbb{X}$. Extend $\{x_1,x_2,\dots,x_k\}$ to a basis $\{x_1,x_2,\dots,x_n\}$ of $\mathbb{X}$ consisting of extreme points. Also, by composing $T$ with a suitable signed permutation operator, we can assume that $x_i=\sum\limits_{j=1}^ne_j-2e_i$, where $e_i$ denotes the $i$-th standard basis vector. \par
Define a map $\Gamma:\mathbb{B}(\mathbb{X},\mathbb{Y})\to\mathbb{Y}_\infty^n$ given by
\begin{align*}
    \Gamma(T'):=(T'x_1,T'x_2,\dots,T'x_n),~T'\in\mathbb{B}(\mathbb{X},\mathbb{Y}).
\end{align*}
Clearly, $\Gamma$ is a bijective bounded linear map having norm 1 and it attains norm at $T$. Now, let $T\perp_BT'$, for some $T'\in\mathbb{B}(\mathbb{X},\mathbb{Y})$. Then by \cite[Theorem 2.2]{Sain2}, there exist $x,x'\in M_T$ such that 
\begin{align*}
    \|Tx+\lambda T'x\|\geq\|Tx\|,~\|Tx'-\lambda T'x'\|\geq\|Tx'\|,~\textit{for every}~\lambda\geq0.
\end{align*}
Now, if $\|Tx+\lambda T'x\|\geq\|Tx\|$ for every $\lambda\geq0$ then either $Tx\perp_BT'x$ or $[Tx,T'x]>0$ for every semi-inner product $[\cdot,\cdot]$ on $\mathbb{Y}$ since otherwise, there exists $\delta>0$ such that $f(T'x)\leq-\delta$ for every support functional $f$ of $Tx$. However, this means that for any support functional $f$ of $Tx$,
\begin{align*}
    \|Tx\|=f(Tx)\geq f(Tx+\lambda T'x)-\delta \lambda.
\end{align*}
Since the support functionals of $Tx$ are limit points of the set of support functionals of $Tx+\lambda T'x$ for every $\lambda>0$, this violates $\|Tx+\lambda T'x\|\geq\|Tx\|$ for every $\lambda\geq0$.\par
Now, if $x\in M_T$, then $x$ lies in the convex hull of $\{x_1,x_2,\dots,x_k\}$. If $S$ is the smallest subset of $\{x_1,x_2,\dots,x_k\}$, such that $x\in\conv(S)$, then $Tx=Tx_i$ for every $x_i\in S$. Hence if $Tx\perp_BT'x$ and $x=\sum\limits_{x_i\in S}\lambda_ix_i$, $\lambda_i>0$, $\sum_i\lambda_i=1$, then find a support functional $f$ of $Tx$ such that $f(T'x)=0$. Since $f$ is a support functional of $Tx_i$ for every $x_i\in S$, 
\begin{align*}
    \Psi:y\mapsto \sum\limits_{x_i\in S}\lambda_if\left(y_i\right),~y=(y_1,y_2,\dots,y_n)\in\mathbb{Y}_\infty^n,
\end{align*}
is a support functional of $\Gamma(T)$ by Proposition \ref{support}, such that $f\left(\Gamma(T')\right)=0$. Also, if $[Tx,T'x]>0$, then $[Tx_i,T'x]>0$ for every $x_i\in S$. Hence 
\begin{align*}
    \sum\limits_{x_i\in S}\lambda_i[T'x_i,Tx_i]>0,
\end{align*}
giving $[Tx_i,T'x_i]>0$ for some $x_i\in S$. Similarly, if $[Tx,T'x]<0$, then $[Tx_i,T'x_i]<0$ for some $x_i\in S$. Now, as $T\perp_BT'$, $Tx\perp_BT'$ for some $x\in M_T$ or $[Tx,T'x]>0$ and $[Ty,T'y]<0$ for every semi-inner product $[\cdot,\cdot]$ on $\mathbb{X}$ for some $x,y\in M_T$. In either case, by Theorem \ref{ortho}, $\Gamma(T)\perp_B\Gamma(T')$.\par
Now, by Theorem \ref{rightsym}, $\Gamma(T)$ is not a right-symmetric point of $\mathbb{Y}_\infty^n$. Further, by Remark \ref{choice}, we can find $y=(y_1,y_2,\dots,y_n)\in\mathbb{Y}_\infty^n$ such that $\Gamma(T)\not\perp_B y$ and $y\perp_B\Gamma(T)$, with $\|y_n\|=\|y\|$ and $\|y_i\|$ arbitrarily small for $i<n$. Now let $x$ be any extreme point of $\mathbb{X}$. Then there exists a partition $(A,B)$ of $\{1,2,\dots,n\}$ such that $x=\sum\limits_{i\in A}e_i-\sum\limits_{i\in B}e_i$. Then if $A=\emptyset$ or $B=\emptyset$,
\begin{align*}
    T'x=\frac{1}{n-2}\sum\limits_{i=1}^nT'x_i, ~\textit{for every}~T'\in\mathbb{B}(\mathbb{X},\mathbb{Y}).
\end{align*}
Else if $|A|,|B|\geq2$, then  
\begin{align*}
    T'x=\frac{|A|-|B|}{2(n-2)}\sum\limits_{i=1}^nT'x_i+\frac{1}{2}\left(\sum\limits_{i\in A}T''x_i-\sum\limits_{i\in B}Tx_i\right),~\forall~T'\in\mathbb{B}(\mathbb{X},\mathbb{Y}).
\end{align*}
Since $n\geq4$, by choosing $\|y_i\|$ sufficiently small, we can ensure that $\Gamma^{-1}(y)$ attains norm only at $\{x_n,-x_n\}$.\par
Therefore, by \cite[Corollary 2.2.1]{Sain2}, $\Gamma^{-1}(y)\perp_BT$. However, as $\Gamma(T)\not\perp_By$, we have $T\not\perp_B\Gamma^{-1}(y)$, violating the right-symmetry of $T$.
\end{proof}

\section{Analysis for the complex case}\label{4}

Our goal in this final section is to prove Theorem \ref{main} over the complex numbers. For analyzing the geometry of the space of operators $\mathbb{B}(\ell_\infty^n,\ell_1^m)$ over $\mathbb{C}$, we follow the same procedure as before. It is trivial to note that the proof of Proposition \ref{support} holds for the space $\mathbb{X}_\infty^n$ over the complex numbers as well. We now restate Theorem \ref{ortho} for the complex case:
\begin{theorem}\label{orthocomplex}
Given $x=(x_1,x_2,\dots,x_n)$ and $y=(y_1,y_2,\dots,y_n)$ elements of $\mathbb{X}_\infty^n$, $x\perp_By$ if and only if 
\begin{align*}
    0\in\conv\{[y_i,x_i]:\|x_i\|=\|x\|,~[\cdot,\cdot]~\textit{semi-inner product on}~\mathbb{X}\}.
\end{align*}
\end{theorem}
\begin{proof}
For the sufficiency, observe that if 
\begin{align*}
    \sum\limits_{k=1}^n\lambda_k[x_k,y_k]_k=0,
\end{align*}
for $\|x_k\|=\|x\|$ and $[\cdot,\cdot]_k$ semi-inner products on $\mathbb{X}$ for $1\leq k\leq n $, then define a support functional of $x$ using Proposition \ref{support} as:
\begin{align*}
    z=(z_1,z_2,\dots,z_n)\mapsto \frac{1}{\|x\|}\sum\limits_{k=1}^n\lambda_k[x_k,z_k]_k,~z\in\mathbb{X}_\infty^n.
\end{align*}
Clearly, the image of $y$ under this support functional is zero, giving $x\perp_By$.\par
For the necessity, assume $x\perp_By$ and consider a support functional of $x$ given by $f=(f_1,f_2,\dots,f_n)$ such that $f(y)=0$. Construct semi-inner products by extending:
\begin{align*}
    [x_i,z_i]_i:=\|x\|g_i(z_i),
\end{align*}
whenever, $f_i=\lambda_i g_i$ with $g_i$ a support functional of $x_i$ and $\lambda_i>0$ as in Proposition \ref{support}. Clearly,
\begin{align*}
    \sum\limits_i\lambda_i[y_i,x_i]_i=0,
\end{align*}
finishing the proof.
\end{proof}
Note that Corollary \ref{smoothinf} holds for the complex case as can be deduced from Theorem \ref{orthocomplex}. The proof of Theorem \ref{rightsym} also holds over $\mathbb{C}$, where positivity of $[y_i,x_i]$ for every semi-inner product is replaced by $\left\{[y_i,x_i]:[\cdot,\cdot]~\text{semi-inner product on}~\mathbb{X}\right\}$ being a compact convex subset of the right half-plane.\par
Theorem \ref{smooth} is restated as:
\begin{theorem}
Given finite-dimensional Banach spaces $\mathbb{X}$ and $\mathbb{Y}$ over $\mathbb{C}$, $T\in\mathbb{B}(\mathbb{X},\mathbb{Y})$ is smooth if and only if $M_T={e^{i\theta}x_0:0\leq\theta<2\pi}$ for some $\|x_0\|=1$ and $Tx_0$ is a smooth point of $\mathbb{Y}$.
\end{theorem}
The proof for Theorem \ref{smooth} holds verbatim in this complex case. It should be noted that these results when applied for the $\mathbb{X}-\mathbb{C}$ case indeed gives the results pertaining to the Birkhoff-James orthogonality in $\ell_\infty$ that were established in \cite{CSS}.\par
Theorems \ref{main} and \ref{leftsym} also hold in the complex case, but for proving them, we need the characterization of Birkhoff-James orthogonality in the space of operators over the complex field.
\begin{theorem}
Given finite-dimensional Banach spaces $\mathbb{X}$ and $\mathbb{Y}$ over $\mathbb{C}$ and $T,S\in\mathbb{B}(\mathbb{X},\mathbb{Y})$, $T\perp_BS$ if and only if
\begin{align*}
    0\in\conv\{[Sx,Tx]:x\in M_T,~[\cdot,\cdot]~\text{semi-inner product on}~\mathbb{X}\}.
\end{align*}
\end{theorem}
\begin{proof}
For the sufficiency, note that if
\begin{align*}
    \sum\limits_{i=1}^k\lambda_i[Sx_i,Tx_i]_i=0,
\end{align*}
then 
\begin{align*}
    S'\mapsto\sum\limits_{i=1}^k\lambda_i[S'x_i,Tx_i]_i,~S'\in\mathbb{B}(\mathbb{X},\mathbb{Y}),
\end{align*}
is a support functional of $T$ that annihilates $S$. \par
For the necessity, note that we can embed $\mathbb{B}(\mathbb{X},\mathbb{Y})$ within $C(S_\mathbb{X},\mathbb{Y})$, the Banach space of all $\mathbb{Y}$ valued continuous functions on $S_\mathbb{X}$ equipped with the supremum norm. Now, applying the Riesz representation theorem, since $\dim(\mathbb{Y})$ is finite, the dual of $C(S_\mathbb{X},\mathbb{Y})$ is given by the space of all $\mathbb{Y}^*$ valued regular Borel measures on $S_\mathbb{X}$ equipped with the total variation norm.\par
Elementary computation yields the characterization of the support functional of a non-zero element $F\in C(S_\mathbb{X},\mathbb{Y})$ as $\mu:S_\mathbb{X}\to\mathbb{Y}^*$ such that
\begin{align*}
    \int\limits_{X}Gd\mu=\int\limits_{M_F}\Psi(x)\left(G(x)\right)d|\mu|(x),~G\in C(S_\mathbb{X},\mathbb{Y}),
\end{align*}
where $\Psi(x)$ is some support functional of $F(x)$ for $x\in M_F$, the set of points in $S_\mathbb{X}$ where $F$ attains its norm and $|\mu|$ is the total variation of the measure $\mu$ with $|\mu|(M_F)=|\mu|(S_\mathbb{X})=1$.\par
Hence if $T\perp_BS$ in $\mathbb{B}(\mathbb{X},\mathbb{Y})$, then $T\perp_BS$ as elements of $C(S_\mathbb{X},\mathbb{Y})$ as well. However, the characterization of the support functionals in $C(S_\mathbb{X},\mathbb{Y})$ gives that $T\perp_BS$ in $C(S_\mathbb{X},\mathbb{Y})$ implies
\begin{align*}
    0\in\conv\{[Tx,Sx]:x\in M_T,~[\cdot,\cdot]~\text{semi-inner product on}~\mathbb{X}\}.
\end{align*}
The proof is similar to the proof of \cite[Theorem 1.1.2]{me}.
\end{proof}
This characterization of Birkhoff-James orthogonality in $\mathbb{B}(\mathbb{X},\mathbb{Y})$ immediately allows us to prove Theorem \ref{leftsym} for the complex case using exactly the same arguments.\par
We finish by proving Theorem \ref{main} for the complex case.
\begin{proof}[Proof of Theorem \ref{main} for the complex case]\hfill
\\
Begin by noting that given any complex Banach space $\mathbb{X}$, and $x,y\in\mathbb{X}$, if $x\not\perp_By$ over $\mathbb{R}$, i.e., $\|x+\lambda y\|<\|x\|$ for some real $\lambda$, then $x\not\perp_By$ in $\mathbb{X}$. Now consider the complex Banach space $\mathbb{B}(\ell_\infty^n,\ell_1^m)$ and $T$ an element of this space, which is not an extreme contraction. Clearly, $T$ cannot attain its norm at $n$ many linearly independent extreme points, by the same argument as in the real case. Further, by composing the operator with a suitable signed permutation operator, we can force the operator to attain its norm at some proper subset of $\{x_i:x_i=\sum\limits_{j=1}^ne_j-2e_i\}$. Once again considering the map $\Gamma:\mathbb{B}(\ell_\infty^n,\ell_1^m)\to\mathbb{X}_\infty^n$ we can construct $y$ such that $y\perp_B\Gamma(T)$ in $\mathbb{R}$. Now suppose $T$ does not attain norm at $x_n$. Repeating the construction in the real case, we can produce $y=(y_1,y_2,\dots,y_n)$ is such a way that $\|y\|=\|y_n\|$ and $y_n\perp_BT(x_n)$ over $\mathbb{R}$. Notice that this forces the existence of a support functional $f$ of $y_n$ having only real components such that $f(Tx_n)=0$. However, by construction, $y_n$ has only real components and therefore, any support functional of $y_n$ in $\mathbb{X}$ over $\mathbb{C}$ has only real components. Therefore, $f$ is in fact a support functional of $y_n$ in $\mathbb{X}$ over $\mathbb{C}$ such that $f(Tx_n)=0$, i.e., in particular, $y_n\perp_BT(x_n)$ in $\mathbb{X}$ and so, $y\perp_B\Gamma(T)$ in $\mathbb{X}_\infty^n$. Now considering $\Gamma^{-1}(y)$ as before, we get that $\Gamma^{-1}(y)\perp_BT$. However, from the proof of the real case, $T\not\perp_B\Gamma^{-1}(y)$ over $\mathbb{R}$ and therefore, on $\mathbb{C}$. This concludes the argument.

\end{proof}


\begin{thebibliography}{100}

\bibitem{gm}
B. Bagchi, G. Misra, \textit{``On Grothendieck Constants"}, \texttt{preprint, 2008}.

\bibitem{B} G. Birkhoff,  \textit{``Orthogonality in linear metric spaces"}, \texttt{Duke Math. J., 1 (1935), 169-172.}


\bibitem{blei}
R. Blei, \textit{``The Grothendieck inequality revisited"}, \texttt{Mem. Amer. Math. Soc. (2014)}.

\bibitem{usseq}
B. Bose, S. Roy, D. Sain, \textit{``Birkhoff-James orthogonality and its local symmetry in some sequence spaces"}, \texttt{arXiv:2205.11586 [math.FA]}

\bibitem{me}
B. Bose, \textit{``Birkhoff-James orthogonality and its point-wise symmetry in some function spaces"}, \texttt{arXiv:2205.13078 [math.FA]}


\bibitem{CSS} A. Chattopadhyay, D. Sain, T. Senapati, \textit{ ``Characterization of symmetric points in $ l_p^n $-spaces"},\texttt{ Linear Multilinear Algebra, 69 (2021), No. 16, 2998-3009}.


\bibitem{groth}
A. Grothendieck, \textit{Resum\'e de la th\'eorie m\'etrique des produits tensoriels topologiques}, \texttt{Boll. Soc. Mat. Sao-Paulo 8 (1953), 1-79, reprinted in Resenhas 2 (1996), No. 4, 401-480.}


\bibitem{iwa}
A. Iwanik, \textit{``Extreme contractions on certain function spaces"}, \texttt{Coll. Math. 40 (1978), No. 1, 147-153}.


\bibitem{james2} R.C. James, \textit{``Inner product in normed linear spaces"}, \texttt{Bull. Amer. Math. Soc., 53 (1947), 559-566.}

\bibitem{james} R.C. James, \textit{``Orthogonality and linear functionals in normed linear spaces"}, \texttt{Trans. Amer. Math. Soc., 61 (1947), 265-292.}


\bibitem{Khot}
S. Khot, A. Naor, \textit{``Grothendieck-type inequalities in combinatorial optimization"}, \texttt{Comm. Pure Appl. Math., 65 (2012), No. 7, 992-1035.}


\bibitem{Krivine}
J.L. Krivine, \textit{``Constantes de Grothendieck et fonctions de type positif sur les sph\'eres"}, \texttt{Adv. in Math. 31 (1979), 16-30}.


\bibitem{lima}
A. Lima, \textit{``The geometric structure of $L(l_\infty^3,l_1^3)$}, \texttt{preprint, 1979.}


\bibitem{Pisier}
G. Pisier, \textit{``Grothendieck's theorem, past and present"}, \texttt{Bull. Amer. Math. Soc. 49 (2012), 237-323}.


\bibitem{BAT}
W. Rudin, \textit{``Functional Analysis", Second Edition}, 
\texttt{Mathematics Series, McGraw-Hill. (1991).}


\bibitem{Sain2} D. Sain, \textit{``Birkhoff-James orthogonality of linear operators on finite dimensional Banach spaces"}, \texttt{J. Math. Anal. Appl., 447 (2017), No. 2,  860-866.}


\bibitem{Sain} D. Sain, \textit{``On the norm attainment set of a bounded linear operator"}, \texttt{J. Math. Anal. Appl., 457 (2018), No. 1, 67-76.}

 
\end{thebibliography}
\end{document}